\documentclass{conm-p-l}
\usepackage{amssymb}

\usepackage{amsfonts}
\usepackage{eucal}
\begin{document}

\font\eightrm=cmr8
\font\eightit=cmti8
\font\eighttt=cmtt8
\def\tci
{\hbox{\hskip1.8pt$\rightarrow$\hskip-11.5pt$^{^{C^\infty}}$\hskip-1.3pt}}
\def\nft
{\hbox{$n$\hskip3pt$\equiv$\hskip4pt$5$\hskip4.4pt$($mod\hskip2pt$3)$}}
\def\bbR{\mathrm{I\!R}}
\def\rto{\bbR\hskip-.5pt^2}
\def\rtr{\bbR\hskip-.7pt^3}
\def\rfo{\bbR\hskip-.7pt^4}
\def\rn{\bbR^{\hskip-.6ptn}}
\def\mr{\bbR^{\hskip-.6ptm}}
\def\bbZ{\mathsf{Z\hskip-4ptZ}}
\def\bbRP{\text{\bf R}\text{\rm P}}
\def\bbC{\text{\bf C}}
\def\hyp{\hskip.5pt\vbox
{\hbox{\vrule width3ptheight0.5ptdepth0pt}\vskip2.2pt}\hskip.5pt}
\def\er{r}
\def\es{s}
\def\df{d\hskip-.8ptf}
\def\fv{\mathcal{F}}
\def\fvp{\fv_{\nrmh p}}
\def\wv{\mathcal{W}}
\def\vt{\mathcal{P}}
\def\tv{\mathcal{T}}
\def\vtx{\vt_{\nh x}}
\def\fh{f}
\def\rc{\theta}
\def\jm{\mathcal{I}}
\def\ke{\mathcal{K}}
\def\xc{\mathcal{X}_c}
\def\lz{\mathcal{L}}
\def\Lie{\pounds}
\def\lv{\Lie\hskip-1.2pt_v\w}
\def\lo{\lz_0}
\def\xe{\mathcal{E}}
\def\eo{\xe_0}
\def\hm{\hskip1.9pt\widehat{\hskip-1.9ptM\hskip-.2pt}\hskip.2pt}
\def\hmt{\hskip1.9pt\widehat{\hskip-1.9ptM\hskip-.5pt}_t}
\def\hmz{\hskip1.9pt\widehat{\hskip-1.9ptM\hskip-.5pt}_0}
\def\hmp{\hskip1.9pt\widehat{\hskip-1.9ptM\hskip-.5pt}_p}
\def\hg{\hskip1.2pt\widehat{\hskip-1.2ptg\hskip-.4pt}\hskip.4pt}
\def\hk{\hskip1.5pt\widehat{\hskip-1.5ptK\hskip-.5pt}\hskip.5pt}
\def\hq{\hskip1.5pt\widehat{\hskip-1.5ptQ\hskip-.5pt}\hskip.5pt}
\def\q{q}
\def\bq{\hat q}
\def\p{p}
\def\w{^{\phantom i}}
\def\x{v}
\def\y{y}
\def\vp{\vt^\perp}
\def\vd{\vt\hh'}
\def\vdx{\vd{}\hskip-4.5pt_x}
\def\bz{b\hh}
\def\fe{F}
\def\fy{\phi}
\def\vl{\Lambda}
\def\hy{\mathcal{V}}
\def\vh{h}
\def\mv{V}
\def\vo{V_{\nnh0}}
\def\ao{A_0}
\def\bo{B_0}
\def\uv{\mathcal{U}}
\def\sv{\mathcal{S}}
\def\svp{\sv_p}
\def\xv{\mathcal{X}}
\def\xvp{\xv_p}
\def\yv{\mathcal{Y}}
\def\yvp{\yv_p}
\def\zv{\mathcal{Z}}
\def\zvp{\zv_p}
\def\cv{\mathcal{C}}
\def\dy{\mathcal{D}}
\def\nv{\mathcal{N}}
\def\iv{\mathcal{I}}
\def\gkp{\Sigma}
\def\ret{\pi}
\def\hs{\hskip.7pt}
\def\hh{\hskip.4pt}
\def\hn{\hskip-.4pt}
\def\nh{\hskip-.7pt}
\def\nnh{\hskip-1pt}
\def\hrz{^{\hskip.5pt\text{\rm hrz}}}
\def\vrt{^{\hskip.2pt\text{\rm vrt}}}
\def\vt{\varTheta}
\def\op{S}
\def\vg{\varGamma}
\def\my{\mu}
\def\ny{\nu}
\def\gy{\lambda}
\def\ax{\alpha}
\def\bx{\beta}
\def\cx{\gamma}
\def\ay{a}
\def\by{b}
\def\cy{c}
\def\gp{\mathrm{G}}
\def\hp{\mathrm{H}}
\def\kp{\mathrm{K}}
\def\gm{\gamma}
\def\Gm{\Gamma}
\def\Lm{\Lambda}
\def\Dt{\Delta}
\def\dg{\Delta}
\def\sj{\sigma}
\def\lg{\langle}
\def\rg{\rangle}
\def\lr{\lg\hh\cdot\hs,\hn\cdot\hh\rg}
\def\vs{vector space}
\def\rvs{real vector space}
\def\vf{vector field}
\def\tf{tensor field}
\def\tvn{the vertical distribution}
\def\dn{distribution}
\def\pt{point}
\def\tc{tor\-sion\-free connection}
\def\ea{equi\-af\-fine}
\def\rt{Ric\-ci tensor}
\def\pde{partial differential equation}
\def\pf{projectively flat}
\def\pfs{projectively flat surface}
\def\pfc{projectively flat connection}
\def\pftc{projectively flat tor\-sion\-free connection}
\def\su{surface}
\def\sco{simply connected}
\def\psr{pseu\-\hbox{do\hs-}Riem\-ann\-i\-an}
\def\inv{-in\-var\-i\-ant}
\def\trinv{trans\-la\-tion\inv}
\def\feo{dif\-feo\-mor\-phism}
\def\feic{dif\-feo\-mor\-phic}
\def\feicly{dif\-feo\-mor\-phi\-cal\-ly}
\def\Feicly{Dif\-feo\-mor\-phi\-cal\-ly}
\def\diml{-di\-men\-sion\-al}
\def\prl{-par\-al\-lel}
\def\skc{skew-sym\-met\-ric}
\def\sky{skew-sym\-me\-try}
\def\Sky{Skew-sym\-me\-try}
\def\dbly{-dif\-fer\-en\-ti\-a\-bly}
\def\cs{con\-for\-mal\-ly symmetric}
\def\cf{con\-for\-mal\-ly flat}
\def\ls{locally symmetric}
\def\ecs{essentially con\-for\-mal\-ly symmetric}
\def\rr{Ric\-ci-re\-cur\-rent}
\def\kf{Killing field}
\def\om{\omega}
\def\vol{\varOmega}
\def\dv{\delta}
\def\ve{\varepsilon}
\def\zt{\zeta}
\def\kx{\kappa}
\def\mf{manifold}
\def\mfd{-man\-i\-fold}
\def\bmf{base manifold}
\def\bd{bundle}
\def\tbd{tangent bundle}
\def\ctb{cotangent bundle}
\def\bp{bundle projection}
\def\prc{pseu\-\hbox{do\hs-}Riem\-ann\-i\-an metric}
\def\prd{pseu\-\hbox{do\hs-}Riem\-ann\-i\-an manifold}
\def\Prd{pseu\-\hbox{do\hs-}Riem\-ann\-i\-an manifold}
\def\npd{null parallel distribution}
\def\pj{-pro\-ject\-a\-ble}
\def\pd{-pro\-ject\-ed}
\def\lcc{Le\-vi-Ci\-vi\-ta connection}
\def\vb{vector bundle}
\def\vbm{vec\-tor-bun\-dle morphism}
\def\kerd{\text{\rm Ker}\hskip2.7ptd}
\def\ro{\rho}
\def\sy{\sigma}
\def\ts{total space}
\def\pmb{\pi}

\newtheorem{theorem}{Theorem}[section] 
\newtheorem{proposition}[theorem]{Proposition} 
\newtheorem{lemma}[theorem]{Lemma} 
\newtheorem{corollary}[theorem]{Corollary} 
  
\theoremstyle{definition} 
  
\newtheorem{defn}[theorem]{Definition} 
\newtheorem{notation}[theorem]{Notation} 
\newtheorem{example}[theorem]{Example} 
\newtheorem{conj}[theorem]{Conjecture} 
\newtheorem{prob}[theorem]{Problem} 
  
\theoremstyle{remark} 
  
\newtheorem{remark}[theorem]{Remark}

\renewcommand{\theequation}{\arabic{section}.\arabic{equation}}

\title[New compact Weyl-par\-al\-lel manifolds]{New examples of compact 
Weyl-par\-al\-lel manifolds}
\author[A. Derdzinski]{Andrzej Derdzinski} 
\address{Department of Mathematics, The Ohio State University, 
Columbus, OH 43210, USA} 
\email{andrzej@math.ohio-state.edu} 
\author[I.\ Terek]{Ivo Terek} 
\address{Department of Mathematics, The Ohio State University, 
Columbus, OH 43210, USA} 
\email{terekcouto.1@osu.edu} 
\subjclass[2020]{Primary 53C50}
\def\leftmark{A.\ Derdzinski \&\ I.\ Terek}
\def\rightmark{New compact Weyl-par\-al\-lel manifolds}

\begin{abstract}
We prove the existence of compact pseu\-do\hs-Riem\-ann\-i\-an manifolds with 
parallel Weyl tensor which are neither con\-for\-mal\-ly flat nor locally
symmetric, and represent all indefinite metric signatures in all dimensions
$\,n\ge5$. Until now such manifolds were only known to exist in dimensions
$\,n=3j+2$, where $\,j\,$ is any positive integer \cite{derdzinski-roter-10}. 
As in \cite{derdzinski-roter-10}, our examples are dif\-feo\-mor\-phic to
nontrivial torus bundles over the circle and arise from a
quo\-tient-man\-i\-fold construction applied to suitably chosen discrete
isometry groups of dif\-feo\-mor\-phi\-cal\-ly-Euclid\-e\-an ``\nh model"
manifolds.
\end{abstract}

\maketitle

\setcounter{section}{0}
\setcounter{theorem}{0}
\renewcommand{\thetheorem}{\Alph{theorem}}
\section*{Introduction}
\setcounter{equation}{0}
{\it Essentially conformally symmetric\/} (briefly, {\it ECS\/}) {\it
manifolds\/} \cite{derdzinski-roter-07} are those 
pseu\-\hbox{do\hskip.7pt-}Riem\-ann\-i\-an manifolds 
of dimensions $\,n\ge4$ which have parallel Weyl tensor 
($\nabla W\hskip-1.7pt=0$) without being con\-for\-mal\-ly flat 
($W\hskip-1.7pt=0$) or locally symmetric ($\nabla\hskip-1.5ptR=0$). Their 
existence, for every $\,n\ge4$, was established by Roter
\cite[Corol\-lary~3]{roter}, who also showed that their metrics are all
indefinite \cite[Theorem~2]{derdzinski-roter-77}. A local description of all
ECS metrics is given in \cite{derdzinski-roter-09}.

Manifolds with $\,\nabla W\hskip-1.7pt=0\,$ are often called {\it
  con\-for\-mal\-ly symmetric\/} \cite{chaki-gupta}. This class represents one of the natural linear conditions imposed on $\nabla R$, cf. \cite[Chapter 16]{besse}, and due to its naturality it attracted the attention of several authors, including, Hotlo\'{s} \cite{hotlos}, Deszcz et al. \cite[Sect.\ 4]{deszcz-glogowska-hotlos-zafindratafa} and Mantica and Suh \cite[Section 3]{mantica-suh}, Deszcz et. al. \cite[Theorem 6.1]{deszcz-glogowska-hotlos-torgasev-zafindratafa}. Results on ECS manifolds, as well as techniques used in obtaining them have been applied to more general classes of manifolds \cite[Example 2.2]{deszcz-hotlos}, \cite{suh-kwon-yang}, \cite[Theorem 3]{alekseevsky-galaev}, \cite{calvino-louzao-garcia-rio-seoane-bascoy-vazquez-lorenzo}, \cite[Theorem 3.9]{calvino-louzao-garcia-rio-vazquez-abal-vazquez-lorenzo}, \cite[Lemma 3]{leistner-schliebner}, \cite{mantica-molinari}, \cite[proofs of Theorems 1.1 and 4.5]{tran} and, more recently, \cite{terek}. 

Every ECS manifold $\,M\,$ carries a distinguished null parallel distribution
$\,\mathcal{D}$ of dimension $\,d\in\{1,2\}$, discovered by Ol\-szak
\cite{olszak}. See also \cite[p.\ 119]{derdzinski-roter-09}. We will refer to
$\,d\,$ as the {\it rank\/} of $\,M\nh$. Explicitly, the sections of
$\,\mathcal{D}\,$ are the vector fields corresponding via the metric to 
$\,1$-forms $\,\xi\,$ such that $\,\xi\wedge[W(v,v'\nh,\,\cdot\,,\,\cdot\,)]=0\,$ for all vector fields $\,v,v'\nh$.

Compact rank-one ECS manifolds which are also geodesically complete are
known to exist \cite[Theorem 1.1]{derdzinski-roter-10} in all dimensions
$\,n\ge5\,$ with $\,n\equiv5\hskip4.4pt(\mathrm{mod}\hskip2.7pt3\hskip.5pt)$.
Whether compactness of an ECS manifold implies its geodesic completeness is an open question. Similarly, it is not known if a compact ECS manifold can have
the dimension $\,n=4$, or rank two, or be locally homogeneous, even though
non\-com\-pact locally homogeneous ECS manifolds of both ranks exist 
\cite{derdzinski-78} for all $\,n\ge4$.

Our main result may be viewed as an improvement on
\cite[Theorem 1.1]{derdzinski-roter-10}, since it covers every dimension
$\,n\ge5$, rather than just those congruent to $\,5\,$ modulo $\,3$.
\begin{theorem}\label{maith}
There exist compact rank-one ECS manifolds of all dimensions $\,n\ge5\,$ and
all indefinite metric signatures, dif\-feo\-mor\-phic to nontrivial torus
bundles over the circle, geodesically complete, and not locally homogeneous.
Their lo\-cal-isom\-e\-try types, in each fixed dimension and metric
signature, form -- in a natural sense -- an in\-fi\-nite-di\-men\-sion\-al
moduli space.
\end{theorem}

The significance of the ``model'' manifolds (see Section \ref{mm}) used to produce our compact examples is twofold: in addition to representing all rank-one local isometry types, they constitute the universal coverings of a large class of compact rank-one ECS manifolds \cite[Theorem 7.1]{derdzinski-roter-08}.

Finally, the bundle structure in Theorem~\ref{maith} reflects a general principle: in
\cite{derdzinski-terek} we show that a 
non-lo\-cal\-ly-ho\-mo\-ge\-ne\-ous compact rank-one ECS manifold, replaced if
necessary by a two-fold isometric covering, {\it must be a bundle over the
circle, having\/ $\,\mathcal{D}^\perp\hskip-2pt$ as the vertical
distribution.}

\section*{Outline of the construction}

The paper is structured as follows. After preliminaries comes 
Section~\ref{mm}, presenting rank-one ECS model manifolds, which exist in all
dimensions $\,n\ge4\,$ (even though our construction of their compact
quotients requires assuming that $\,n\ge5$). The purpose of Section~\ref{gc}
is to show that the particular model manifolds which we focus on are
geodesically complete, but not locally homogeneous, so that the same
conclusion holds for our compact quotients. In Section~\ref{gl} we observe
(Lemma~\ref{polyn}) that, given an integer $\,m\ge3$, there exists a 
$\,\mathrm{GL}\hh(m,\bbZ)\,$ polynomial with $\,m\,$ distinct real positive
roots different from $\,1$. The next two sections, crucial for our existence
argument, are devoted to proving, in Theorem~\ref{pospe}, that all the
$\,m$-el\-e\-ment $\,\mathrm{GL}\hh(m,\bbZ)$-spec\-tra just mentioned -- and
even a wider class characterized by condition (\ref{pol}) -- arise via a
specific integral formula from periodic curves $\,\bbR\ni t\mapsto B=B(t)\,$
of diagonal $\,m\times m\,$ matrices satisfying an ordinary differential
equation of the form $\,\dot B\hs+B^2\nh=\fh\hs+A$, with a function $\,\fh\,$
and matrix $\,A$ appearing in a suitable rank-one ECS model manifold
$\,\hm\,$ of dimension $\,n=m+2$. Finally, Section~\ref{po} provides the
existence proof: a curve $\,t\mapsto B(t)\,$ realizing, for any given
$\,m\ge3$, one of the $\,\mathrm{GL}\hh(m,\bbZ)$-spec\-tra of
Lemma~\ref{polyn}, is used to construct a group $\,\Gm\,$ acting on the
corresponding model $\,\hm\,$ freely and properly dis\-con\-tin\-u\-ous\-ly 
by isometries, with a compact quotient manifold $\,M=\hm\nnh/\hh\Gm\nh$.

The most important ingredient of the above argument is the
$\,\mathrm{GL}\hh(m,\bbZ)$-spec\-trum property of the curve
$\,t\mapsto B(t)$.

\renewcommand{\thetheorem}{\thesection.\arabic{theorem}}
\section{Preliminaries}\label{pr} 
\setcounter{equation}{0}
By a {\it lattice\/} in a real vector space $\,\lz\,$ with 
$\,\dim\lz=m<\infty\,$ we mean, as usual, an additive subgroup of $\,\lz\,$ 
generated by some basis of $\,\lz$. Then
\begin{equation}\label{dsc}
\Lm\,\mathrm{\ is\ a\ discrete\ subset\ of}\,\lz\hs,
\end{equation}
as one sees identifying $\,\Lm\,$ and $\,\lz\,$ with $\,\bbZ^m$ and $\,\mr\nh$.

Suppose that a group $\,\Gm\hs$ act on a manifold $\,\hm\,$ freely by \feo s.
The action of $\,\Gm\hs$ on $\,\hm\,$ is called {\it 
properly dis\-con\-tin\-u\-ous\/} if there exists a locally \feic\ surjective 
mapping $\,\hm\to M\,$ onto some manifold $\,M\nh$, the pre\-im\-ages of
points of $\,M\,$ under which are precisely the orbits of the $\,\Gm\hs$
action. 
\begin{remark}\label{stbil}A free left action of a group $\,\Gm\hs$ on a
manifold $\,\hm\,$ is properly dis\-con\-tin\-u\-ous if and only if for any 
sequences $\,a\hn_j\w$ in $\,\Gm\hs$ and $\,y\hn_j\w$ in $\,\hm\nh$, with
$\,j\,$ ranging over positive integers, such that
both $\,y\hn_j\w$ and $\,a\hn_j\w y\hn_j\w$ converge, the sequence
$\,a\hn_j\w$ is constant except for finitely many $\,j$. See
\cite[Exercise 12\hs-\nh19 on p.\ 337]{lee}.
\end{remark}
\begin{remark}\label{bndpr}
Any smooth sub\-mer\-sion from a compact manifold into a connected manifold is
a (surjective) bundle projection. This is the compact case of Ehres\-mann's
fibration theorem \cite[Corollary 8.5.13]{dundas}.
\end{remark}

\section{The model manifolds}\label{mm} 
\setcounter{equation}{0}
Let $\,\fh\nh,p,n,\mv\nh,\lr\,$ and $\,A\,$ denote a nonconstant
$\,C^\infty$ function $\,\fh\nh:\bbR\to\bbR$, periodic of period $\,p>0$, an
integer $\,n\ge4$, a pseu\-\hbox{do\hs-}Euclid\-e\-an inner product $\,\lr$
on a \rvs\ $\,\mv\hh$ of dimension $\,n-2$, and a
nonzero, traceless, $\,\lr$-self-ad\-joint linear en\-do\-mor\-phism
$\,A\,$ of $\,\mv\nnh$. \hbox{Consider the \prc\ \cite{roter}}
\begin{equation}\label{met}
\kappa\,dt^2\nh+\,dt\,ds\hs+\hs\delta
\end{equation}
on the manifold $\,\hm\nh=\rto\nh\times\mv\hs\approx\,\rn\nh$. The products of
differentials stand here for symmetric products, $\,t,s\,$ are the Cartesian
coordinates on $\,\rto$ treated, with the aid of the projection
$\,\hm\to\rto\nh$, as functions
$\,\hm\to\bbR$, and $\,\delta\,$ is the pull\-back to $\,\hm\,$ of the flat
(constant) \prc\ on $\,\mv$ arising from the inner product $\,\lr$, while
$\,\kx:\hm\to\bbR\,$ is the function given by 
$\,\kx(t,s,x)=f(t)\hskip.4pt\langle x,x\rangle+\langle Ax,x\rangle$. 

The metric (\ref{met}) turns $\,\hm\,$ into a rank-one ECS manifold
\cite[Theorem~4.1]{derdzinski-roter-09}.

We now define $\,\xe\hs$ to be the vector space of all $\,C^\infty$ solutions
$\,u:\bbR\to\mv\,$ to the differential equation
$\,\ddot u(t)=\fh(t)\hh u(t)+Au(t)$, and set
$\,\gp=\bbZ\times\bbR\times\xe\nh$. Whenever $\,u,w\in\xe\nh$, the function 
$\,\varOmega(u,w)=\lg\dot u,w\rg-\lg u,\dot w\rg:\bbR\to\bbR\,$ is constant,
giving rise to a nondegenerate \skc\ bi\-lin\-e\-ar form $\,\varOmega\,$ on
$\,\xe\nh$. We also have a natural linear isomorphism $\,T:\xe\to\xe\,$ with 
$\,(Tu)(t)=u(t-p)$. Next, we turn $\,\gp\,$ into a Lie group by declaring the
group operation to be
\begin{equation}\label{ope}
(k,q,u)\cdot(\ell,r,w)
=(k+\ell,\hs q+r-\varOmega(u,T^{\hs \ell\nh}w),\hs T^{\hs-\hs \ell\nh}u+w)\hh,
\end{equation}
and introduce a left action of the Lie group $\,\gp\,$ on the manifold
$\,\hm\nh$, with
\begin{equation}\label{act}
(k,q,u)\cdot(t,s,v)
=(t+kp,\hs s+q-\lg\dot u(t),2v+u(t)\rg,\hs v+u(t))\hs.
\end{equation}
With all triples assumed to be elements of $\,\gp$, one then has
\begin{equation}\label{cnj}
(k,q,u)\cdot(0,r,w)\cdot(k,q,u)^{-1}\nh
=(0,r-2\hh\varOmega(u,w),T^{\hs k}w)\hh,
\end{equation}
Our $\,\gp\,$ also acts on the manifold $\,\rto\nh\times\xe\hs$, \feic\ 
to $\,\bbR^{\hskip-.5pt2n-2}\nnh$, by
\begin{equation}\label{qkx}
(k,q,u)\cdot(t,z,w)\,=\,(t+kp,\,z+q-\varOmega(u,w),\,T^k(w+u))\hs,
\end{equation}
and the following mapping is equi\-var\-i\-ant relative to the
$\,\gp$-ac\-tions (\ref{qkx}) and (\ref{act}):
\begin{equation}\label{mpg}
\rto\nh\times\xe\ni(t,z,w)\mapsto(t,s,v)
=(t,\,z-\lg\dot w(t),w(t)\rg,\,w(t))\in\hm\hs.
\end{equation}
All the above facts are established in \cite[p.\ 77]{derdzinski-roter-10},
where it is also shown that
\begin{equation}\label{ism}
\mathrm{the\ group\ }\,\gp\hs\mathrm{\ acts\ on\ }\,\hm\,\mathrm{\ by\
isometries\ of\ the\ metric\ (\ref{met}).}
\end{equation}
\begin{remark}\label{relax}If, at the beginning of this section,
$\,\fh\nh:I\to\bbR\,$ is just a nonconstant $\,C^\infty$ function on an open 
interval $\,I\subseteq\bbR$, rather than being defined on $\,\bbR$ and
periodic, while the remaining data $\,n,\mv\nh,\lr\,$ and $\,A\,$ are as
before, $\,I\nh\times\bbR\times\mv$ with the metric (\ref{met})
will still be a rank-one ECS manifold; conversely, in any
$\,n$-di\-men\-sion\-al rank-one ECS manifold, every point at which the
covariant derivative of the Ric\-ci tensor is nonzero has a neighborhood
isometric to an open su\-bset of a manifold of this type
\cite[Theorem~4.1]{derdzinski-roter-09}.
\end{remark}

\section{Geodesic completeness}\label{gc} 
\setcounter{equation}{0}
Later in this section, showing that local homogeneity implies relation 
(\ref{dfm}), we use much weaker assumptions than necessary for the purposes
of the present paper. The reason is that we need to cite such a general
conclusion when proving a result in another paper, namely, 
\cite[Theorem 6.3]{derdzinski-terek}.

For the manifold $\,\hm\nh=\rto\nh\times\mv\hs\approx\,\rn$ with the metric
(\ref{met}),
\begin{equation}\label{gco}
\hm\,\mathrm{\ is\ geodesically\ complete,\ but\ not\ locally\ homogeneous.}
\end{equation}
To see this, we let $\,i,j\,$ range over $\,2,\dots,n-1$, fix linear 
coordinates $\,x^i$ on $\,\mv$ which, along with $\,x^1\nh=t\,$ and
$\,x^n\nh=s/2$, form a global coordinate system on $\,\hm\nh$. The
pos\-si\-bly-non\-ze\-ro components of (\ref{met}), its reciprocal metric,
and the Le\-vi-Ci\-vi\-ta connection $\,\nabla\hs$ then are those
algebraically related to
\begin{equation}\label{pnz}
\begin{array}{l}
g_{11}\w=\kx\hh,\hskip6ptg_{1n}\w=1\hh,\hskip6pt
g^{1n}=1\hh,\hskip6ptg^{nn}=-\kx\hh,\hskip6pt
\mathrm{(constants)\ }\,g_{ij}\w\mathrm{\ and\ }\,g^{ij}\hh,\\
\vg_{\hskip-2.7pt11}^{\hs n}=\partial\nh_1\w\kx/2\hh,\hskip28pt
\vg_{\hskip-2.7pt11}^{\hs i}=-g^{ij}\partial\nnh_j\w\kx/2\hh,
\hskip28pt
\vg_{\hskip-2.7pt1i}^{\hs n}=\partial\nh_i\w\kx/2\hh.
\end{array}
\end{equation}
See \cite[p.\ 93]{roter}. With $\,(\hskip2.3pt)\hskip-2.2pt\dot{\phantom o}$
referring to the geodesic parameter, the geodesic equations read 
$\,\ddot x^1\nh=0\,$ (and so $\,c=\dot x^1$ is constant), 
$\,\ddot x\hh^i\nh=-c^2\vg_{\hskip-2.7pt11}^{\hs i}$ and 
$\,\ddot x^n\nh=-c^2\vg_{\hskip-2.7pt11}^{\hs n}
-2c\vg_{\hskip-2.7pt1i}^{\hs n}\dot x^i\nh$. Note that $\,\kx\,$ has the
form $\,\kx=\fh(t) g_{ij}\w x\hh^ix^j\nh+a_{ij}\w x\hh^ix^j\nh$, with constants
$\,g_{ij}\w$ and $\,a_{ij}\w$. Being linear in the
parameter of a maximal geodesic, $\,x^1$ is defined on $\,\bbR$, and 
the same follows for all $\,x\hh^i$ (as they now satisfy a system of 
linear sec\-ond-or\-der equations) and for $\,x^n$ (which then has a
prescribed second derivative). Thus, completeness follows. The second claim of
(\ref{gco}) is a consequence of what we show later in this
section: namely, more generally, for any 
metric on $\,I\nh\times\bbR\times\mv\hs$ as in Remark~\ref{relax}, with an
open interval $\,I\subseteq\bbR$, if $\,(t,s,x)\in I\nh\times\bbR\times\mv\hs$
has a prod\-uct-type neighborhood $\,\,U\hs$ involving a sub\-in\-ter\-val
$\,I\nh'$ of $\,I\hs$ with a Kil\-ling field $\,v\,$ on $\,\,U$ such 
that $\,d_v\w\hh t\ne0\,$ everywhere in $\,\,U$ (or, equivalently, 
$\,v^1\nh\ne0\,$ everywhere in $\,\,U$ for the above coordinates
$\,x^1\nh,\dots,x^n$), then
\begin{equation}\label{dfm}
(|\fh|^{-\nnh1/2})\hskip-2.2pt\dot{\phantom o}\mathrm{\ is\ constant\ on\
every\ subinterval\ of\ }\,I\hn'\mathrm{\ on\ which\ }\,|\fh|>0\hh,
\end{equation}
$(\hskip2.3pt)\hskip-2.2pt\dot{\phantom o}$ this time denoting $\,d/dt$, and
(\ref{dfm}) in turn easily yields
\begin{equation}\label{yie}
\mathrm{positivity\ of\ }\,|\fh|\,\mathrm{\ and\ constancy\ of\ 
}\,(|\fh|^{-\nnh1/2})\hskip-2.2pt\dot{\phantom o}\mathrm{\ on\ }\,I\nh'\nh,
\end{equation}
since a maximal open subinterval $\,I\nh'\nh'$ of $\,I\nh'$ with $\,|\fh|>0\,$
on $\,I\nh'\nh'$ must equal $\,I\nh'$ (or else $\,I\nh'\nh'$ would have a
finite endpoint lying in $\,I\nh'$, at which $\,|\fh|^{-\nnh1/2}\nh$, being a
linear function, would have a finite limit, contrary to maximality of 
$\,I\nh'\nh'\nh$). Local homogeneity of (\ref{met}) on
$\,\hm\nh=\rto\nh\times\mv\hs$ would, by (\ref{yie}) for $\,I\nh'\nh=\bbR$,
imply that the nonconstant periodic function $\,|\fh|^{-\nnh1/2}$ is linear.
This contradiction proves (\ref{gco}).

We now establish (\ref{dfm}), assuming what is stated in the five
lines preceding (\ref{dfm}), and using the coordinates $\,x^1\nh,\dots,x^n$
on $\,I\nh\times\bbR\times\mv\hs$ mentioned there. (Thus --
cf. Remark~\ref{relax} -- we are looking at an {\it arbitrary\/ 
$\,n$-di\-men\-sion\-al rank-one ECS manifold}, rather than just $\,\hm\,$ 
with the metric (\ref{met}), where $\,\fh\,$ is periodic.) First,
\begin{equation}\label{inv}
\begin{array}{l}
\mathrm{the\ function\ }\,t=x^1\mathrm{\ is\ determined,\ uniquely\ up\ to\
af\-fine\ substitu}\hyp\\
\mathrm{tions,\ by\ the\ local\ geometry\ of\ the\ metric\ (\ref{met}),\ while\
the\ assign}\hyp\\
\mathrm{ment\ }\,t\mapsto\fh(t)\mathrm{,modulo\ its\ replacements\ by\ }\,t\mapsto q^2\nnh\fh(qt+p)\mathrm{,\
with}\\
q,p\in\bbR\,\mathrm{\ and\ }\,q\ne0\mathrm{,\ is\ a\ 
local\ geometric\ invariant\ of\ (\ref{met})\ as\ well.}
\end{array}
\end{equation}
In fact, by (\ref{pnz}), the coordinate vector field $\,\partial\nh_n\w$ is
parallel. Hence so is the $\,1$-form $\,dt=dx^1$ corresponding to 
$\,\partial\nh_n\w$ via the metric (\ref{met}). According to 
\cite[p.\ 93]{roter}, where the convention about the sign of the curvature
tensor is the opposite of ours,
\begin{equation}\label{ric}
\mathrm{the\ metric\ (\ref{met})\ has\ the\ Ric\-ci\ tensor\
}\,\mathrm{Ric}=(2-n)\fh(t)\,dt\otimes dt\hh,
\end{equation}
and the only
pos\-si\-bly-non\-ze\-ro components of its Weyl tensor $\,W\hs$ are those
algebraically related to $\,W\hskip-2.5pt_{1i1\nh j}\w$. Thus, for any vector 
fields $\,v,v'\nh$, the $\,2$-form $\,W(v,v'\nh,\,\cdot\,,\,\cdot\,)$ is 
$\,\wedge$-di\-vis\-i\-ble by $\,dt=dx^1$ and, consequently, the parallel
gradient $\,\partial\nh_n\w=\nabla\hn t\,$ spans the Ol\-szak distribution
$\,\mathcal{D}\,$ described in the Introduction. This yields the first claim
of (\ref{inv}), while the second one is then obvious from (\ref{ric}). By
(\ref{inv}), the local flow of our Kil\-ling 
field $\,v\,$ on $\,\,U\nh$, with $\,d_v\w\hh t\ne0$, sends $\,t\,$ to
af\-fine functions of $\,t$, and so $\,d_v\w\hh t=\lv\hh t=at+b$, where 
$\,a,b\in\bbR\,$ and $\,(a,b)\ne(0,0)$. Thus,
$\,\lv\hs dt=d\lv\hh t=a\hs dt$
and $\,\lv[\fh(t)]=d_v\w[\fh(t)]=\dot\fh(t)\,d_v\w\hh t=(at+b)\dot\fh(t)$. 
From (\ref{ric}) and the Leib\-niz rule, 
$\,0=\lv[\fh(t)\,dt\otimes dt]=[(at+b)\dot\fh(t)+2a\fh(t)]\,dt\otimes dt$. As
$\,(a,b)\ne(0,0)$, (\ref{dfm}) follows.
\begin{remark}\label{infdm}Let $\,\gp\,$ be the group defined in
Section~\ref{mm}, acting on $\,\hm\,$ via (\ref{act}). Due to (\ref{inv}), if
a family of metrics arises from (\ref{met}) 
on quotient manifolds $\,M=\hm\nnh/\hh\Gm\,$ of a fixed dimension $\,n\ge5$,
for sub\-groups $\,\Gm\subseteq\gp\,$ acting on $\,\hm\,$ freely and properly
dis\-con\-tin\-u\-ous\-ly, where $\,\fh\,$ 
used in (\ref{met}) ranges over an in\-fi\-nite-di\-men\-sion\-al manifold of 
$\,C^\infty$ functions of a given period $\,p$, then such a family of metrics
forms an in\-fi\-nite-di\-men\-sion\-al moduli space of lo\-cal-isom\-e\-try
types.
\end{remark}

\section{$\mathrm{GL}\hh(m,\bbZ)\,$ polynomials}\label{gl} 
\setcounter{equation}{0}
By a $\,\mathrm{GL}\hh(m,\bbZ)\,$ {\it polynomial\/} we mean here any degree
$\,m\,$ polynomial with integer coefficients, leading coefficient
$\,(-\nnh1)^m\nh$, and constant term\ $\,1\,$ or $\,-\nnh1$. It is well known,
cf.\ \cite[p.\ 75]{derdzinski-roter-10}, that these are precisely the
characteristic polynomials of matrices in $\,\mathrm{GL}\hh(m,\bbZ)$, the
group of invertible elements in the ring of $\,m\times m$ matrices with
integer entries. Equivalently, the $\,\mathrm{GL}\hh(m,\bbZ)\,$ polynomials
are
\begin{equation}\label{glz}
\begin{array}{l}
\mathrm{the\ characteristic\ polynomials\ of\ en\-do\-mor\-phisms\ of\ an\ }
\hs\,m\hyp\mathrm{di\-men}\hyp\\
\mathrm{sion\-al\ real\ vector\ space\ }\hs V\nh\mathrm{\ sending\ some\
lattice\ }\hs\Lm\hs\mathrm{\ in\ }\hs V\nh\mathrm{\ onto\ itself.}
\end{array}
\end{equation}
On $\,(\gy_1\w,\dots,\gy_m\w)\in\mr$ one may impose the following condition:
\begin{equation}\label{pol}
\begin{array}{l}
\{\gy_1\w,\dots,\gy_m\w\}\,\mathrm{\ is\ a\ subset\ of\ 
}\,(0,\infty)\smallsetminus\{1\}\mathrm{,\ not}\\
\mathrm{of\ the\ form\ }\,\{\gy\}\,\mathrm{\ or\ }\,\{\gy,\gy\nh^{-\nnh1}\nh\}\,\mathrm{\
with\ any\ }\,\gy>0\hh.
\end{array}
\end{equation}
This amounts to requiring that
$\,\gy_1\w,\dots,\gy_m\w\in(0,\infty)\,$ and the absolute values 
$\,|\nh\log\gy_1\w\hn|,\dots,|\nh\log\gy_m\w\hn|\,$ be all positive, 
but not all equal. Clearly,
\begin{equation}\label{cns}
\mathrm{if\
}\{\gy_1\w,\dots,\gy_m\w\}\subseteq(0,\infty)\smallsetminus\{1\}\,\mathrm{\
has\ more\ than\ two\ elements,\ (\ref{pol})\ follows.}
\end{equation}
\begin{lemma}\label{polyn}For every integer\/ $\,m\ge3\,$ there exists a\/ 
$\,\mathrm{GL}\hh(m,\bbZ)\,$ polynomial, the roots\/ 
$\,\gy_1\w,\dots,\gy_m\w$ of which are all real, positive, distinct, and
different from\/ $\,1$.
\end{lemma}
\begin{proof}For $\,m=3$, according to \cite[Lemma 2.1]{derdzinski-roter-10}, 
whenever $\,k,\ell\in\bbZ\,$ and $\,2\le k<\ell\le k^2\nnh/4$, the polynomial 
$\,\gy\mapsto-\gy\nh^3\nh+k\gy\nh^2\nh-\ell\gy+1\,$ has three distinct real
roots $\,\gy,\mu,\nu\,$ with $\,1/\ell<\gy<1<\mu<k/2<\nu<k$, as required.
Since the quadratic polynomial $\,\gy\mapsto\gy\nh^2\nh+k\gy+1\,$
with any integer $\,k<-\nh2\,$ has one root in $\,(0,1)\,$ and another in
$\,(0,\infty)$, both of them depending on $\,k\,$ via an injective function,
products of such quadratic polynomials with different values of $\,k\,$
realize our claim for all even $\,m$, while the case of odd $\,m\,$ is
settled by the same products, further multiplied by the above cubic
polynomial.
\end{proof}
The quartic polynomials obtained in the above proof have very special sets
of roots, of the form $\,\{\gy,\gy\nh^{-\nnh1}\nnh,\mu,\mu^{-\nnh1}\nh\}$. To
obtain more diverse spectra, consider
$\,\gy\mapsto P(\gy)=\gy\nh^4\nh-m\gy\nh^3\nh+\ell\gy\nh^2\nh-k\gy+1$,
with $\,k,\ell,m\in\bbZ$. The inequalities
\[
P(2)\,\le\,16\hs P(1/2)\,<\,0\,<\,P(1)
\]
sufficient for our conclusion will follow once $\,k\ge7$, as we may then choose
$\hs m\hs$ with $\,k\le m\le2k-7\,$ (and so $\,16\hs P(1/2)-P(2)=6(m-k)\ge0$).
This also gives $\,2(k+m)-2<4k+m-8$, and hence $\,2(k+m)-4<2\ell<4k+m-8\,$ for
some $\,\ell$, which amounts to $\,P(1)>0>16\hs P(1/2)$.

\section{Smooth\-ness-pre\-serv\-ing retractions}\label{sr} 
\setcounter{equation}{0}
The following fact will be needed in Section~\ref{na}. A {\it retraction\/}
from a set onto a subset is, as usual, a mapping equal to the identity on the
subset, and by the {\it function component\/} of a pair $\,(\fh\nh,A)\,$ we
mean the $\,C^k$ function $\,\fh\nh$.
\begin{lemma}\label{smthg}
Let there be given an integer\/ $\,k\ge0$, a compact smooth manifold\/ $\,Q$,
fi\-nite-di\-men\-sion\-al real vector spaces\/
$\,\mathcal{X}\,$ and\/ $\,\mathcal{Z}$, an open set\/
$\,\mathcal{U}\nh'\subseteq C^k\nh(Q,\bbR)\times\mathcal{X}$, a smooth mapping 
$\,\varPhi:\mathcal{U}\nh'\to\mathcal{Z}$, and a point\/ 
$\,y\in\mathcal{U}\nh'\nh$. If the differential\/ $\,d\hh\varPhi\nnh_y\w$ is 
surjective, and\/ $\,z=\varPhi(y)$, then there exist a neighborhood\/
$\,\mathcal{U}\,$ of\/ $\,y\,$ in\/ $\,\mathcal{U}\nh'\,$ and a smooth 
retraction\/ $\,\ret:\mathcal{U}\to\mathcal{U}\cap\varPhi^{-\nnh1}\nh(z)\,$ 
such that, for every\/ $\,x\in\mathcal{U}\,$ having a smooth function
component, the function component of\/ $\,\ret(x)\,$ is smooth as well.
\end{lemma}
\begin{proof}Let
$\,x_{\nh1}\w,\dots,x_{\nh m}\w\in C^k\nh(Q,\bbR)\times\mathcal{X}\,$
be representatives of a basis of the quotient space 
$\,[C^k\nh(Q,\bbR)\times\mathcal{X}]/\mathrm{Ker}\,d\hh\varPhi\nnh_y\w\hs 
\approx\,\mathcal{Z}$, having smooth function components. The smooth mapping
$\,F:C^k\nh(Q,\bbR)\times\mathcal{X}\times\mr\to\mathcal{Z}\,$ sending
$\,(x,a^1\nh,\dots,a^m)\,$ to
$\,\varPhi(x+a^i\nh x_{\nh i}\w)$, with
$\,x\in C^k\nh(Q,\bbR)\times\mathcal{X}\,$ and summation over $\,i=1,\dots,m$,
has the value $\,z\,$ at $\,(y,0,\dots,0)$, while the differential at 
$\,(y,0,\dots,0)\,$ of the restriction of $\,F\,$ to $\,\{y\}\times\mr\nh$, 
given by $\,(a^1\nh,\dots,a^m)\mapsto\hs
d\hh\varPhi\nnh_y\w(a^i\nh x_{\nh i}\w)$, is an iso\-mor\-phism due to our
choice of $\,x_{\nh1}\w,\dots,x_{\nh m}\w$ as a
quotient basis. The implicit mapping
theorem \cite[p.\ 18]{lang} thus provides neighborhoods
$\,\mathcal{U}\,$ of $\,y\,$ in $\,\mathcal{U}\nh'\,$ and
$\,\mathcal{U}\nh'\nh'$
of $\,0\,$ in $\,\mr$ such that every $\,x\in\mathcal{U}$ has 
$\,F(x,a^1(x),\dots,a^m(x))=z\,$ with a unique
$\,(a^1(x),\dots,a^m(x))\in\mathcal{U}\nh'\nh'\nh$, which also depends
smoothly on $\,x$. We may now set 
$\,\ret(x)=x+a^i(x)x_{\nh i}$.
\end{proof}

\section{Nearly-arbitrary positive spectra}\label{na} 
\setcounter{equation}{0}
Given $\,p\in(0,\infty)\,$ and an integer $\,m\ge3$, we denote by
$\,\dg\nh_m\w$ the space of all diagonal $\,m\times m\,$
matrices with real entries. Let us consider 
the set $\,\mathcal{X}_m\w$ of all ordered $\,(m+1)$-tuples 
$\,(\fh\nh,\gy_1\w,\dots,\gy_m\w)\,$ formed by a nonconstant periodic
$\,C^\infty$ function $\,\fh:\bbR\to\bbR\,$ of period $\,p\,$ and positive
real numbers $\,\gy_1\w,\dots,\gy_m\w$ such that, for some nonzero trace\-less
matrix $\,A\in\dg\nh_m\w$ and some $\,C^\infty$ curve 
$\,\bbR\ni t\mapsto B=B(t)\in\dg\nh_m\w$,
periodic of period $\,p$, one has
\begin{equation}\label{prs}
\dot B\hs\,+\,B^2\hs=\,\fh\hs\,+\,A\hskip8pt\mathrm{and}\hskip8pt
\mathrm{diag}\hs(\log\gy_1\w,\dots,\log\gy_m\w)\,
=\,-\hskip-4pt\int_{\nh0}^p\hskip-4ptB(t)\,dt\hh,
\end{equation}
where $\,(\hskip2.3pt)\hskip-2.2pt\dot{\phantom o}\nh=d/dt\,$ and $\,\fh\,$
stands for $\,\fh\,$ times $\,\mathrm{Id}\,$ or, equivalently, 
for $\,\mathrm{diag}\hs(\fh\nh,\dots,\fh\nh)$.

In the remainder of this section, we fix an integer $\,k\ge1\,$ and treat
real or ma\-trix-val\-ued functions of period $\,p\,$ as mappings with the
domain $\,S^1\nh$.
\begin{remark}\label{isomo}
If $\,c\in\bbR\smallsetminus\{0\}$, the linear operator 
$\,C^k\nh(S^1\nh,\bbR)\to C^{k-1}\nh(S^1\nh,\bbR)$ sending 
$\,y\,$ to $\,\dot y+cy\,$ is an isomorphism: its kernel consists of multiples
of $\,t\mapsto e^{-ct}\nh$, while solving the equation $\,\dot y+cy=u\,$ with
$\,u\in C^{k-1}\nh(S^1\nh,\bbR)\,$ for $\,y:\bbR\to\bbR\,$ gives us
$\,y(t+p)=y(t)+a\hh e^{-ct}\nh$, where $\,a\in\bbR$. Now 
$\,t\mapsto y(t)+(1-e^{-pc})^{-\nnh1}\nh a\hh e^{-ct}$ is the unique periodic
solution to $\,\dot y+cy=u$.
\end{remark}
\begin{theorem}\label{pospe}
If\/ $\,\gy_1\w,\dots,\gy_m\w$ satisfy\/ {\rm(\ref{pol})}, then\/
$\,(\fh\nh,\gy_1\w,\dots,\gy_m\w)\in\mathcal{X}_m\w$ for all\/ 
$\,\fh\,$ from some in\-fi\-nite-di\-men\-sion\-al manifold of\/ $\,C^\infty$ functions.
\end{theorem}
\begin{proof}At any nonsingular 
$\hs C\nh=\hn\mathrm{diag}\hs(c_1\w,\dots,c_m\w)\in\dg\nh_m\w$
such that $\hs|\hh c_1\w\hn|,\dots,|\hh c_m\w\hn|$ are not all equal, 
viewed as a constant mapping $\,C:S^1\to\dg\nh_m\w$, the
$\,C^\infty$ mapping 
$\,\op:C^k\nh(S^1\nh,\dg\nh_m\w)\to
C^{k-1}\nh(S^1\nh,\dg\nh_m\w)\,$ with $\,\op(B)=\dot B+B^2$ has
the differential given by 
$\,d\hs\op\nh_C\w Y=\dot Y\nh+2CY$, which is an isomorphism 
(Remark~\ref{isomo}). Let $\,C^2\nh=\hs h+A$, with $\,h\,$ a (constant)
multiple of $\,\mathrm{Id}\,$ and (nonzero) trace\-less $\,A$. Due to the
inverse mapping theorem \cite[p.\ 13]{lang}, $\,\op\hs$ has a local
$\,C^\infty$ inverse from a $\,C^{k-1}\nh$-neigh\-bor\-hood of $\,h+A$ onto a 
$\,C^k\nh$-neigh\-bor\-hood of $\,C$. If $\,\fh\in C^{k-1}\nh(S^1\nh,\bbR)\,$
is $\,C^{k-1}\nh$-close to the constant $\,h$, and $\,E\in\dg\nh_m\w$ constant,
trace\-less as well as close to $\,A$, applying to $\,\fh+E\,$ this local
inverse followed
by the mapping $\,B\mapsto(\gy_1\w,\dots,\gy_m\w)\in\mr$ characterized by the
second part of (\ref{prs}), we get the composite 
$\,\fh+E\mapsto\varPhi(\fh+E)=(\gy_1\w,\dots,\gy_m\w)\in\mr$ 
of three mappings: first, the above local inverse of $\,\op\,$
(restricted to the set $\,\mathcal{U}\nh'\,$ of $\,\fh+E\,$ with $\,\fh\,$ near
$\,h\,$ and constant trace\-less $\,E\,$ near $\,A$), then the linear operator
$\,B\mapsto-\nnh\int_0^p\nnh B(t)\,dt\in\dg\nh_m\w$ and, finally,
the entrywise exponentiation of diagonal $\,m\times m\,$ matrices. 
The differential $\,d\hh\varPhi\nnh_{\hs h+A}\w$ is thus the composite
\begin{equation}\label{cmp}
\fh+E\,\mapsto\,Y\mapsto Z=-\hskip-4pt\int_{\nh0}^p\nnh Y(t)\,dt\,
\mapsto\,e^{-pC}\nnh Z 
\end{equation}
of the differentials of our three mappings, at the points $\,h+A,\,C\,$ and
$\,-\nnh\int_0^p\nh C\,dt=-p\hs C$. Note that the first arrow in (\ref{cmp})
sends $\,\fh+E\,$ to $\,Y\nh$ with $\,\dot Y\nh+2CY=\fh+E$, while the entrywise
exponentiation has at $\,-pC\,$ the differential
$\,Z\nh\mapsto e^{-pC}\nnh Z$. Integrating $\,\dot Y\nh+2CY=\fh+E\,$ from
$\,0\,$ to $\,p$, we obtain
$\,2C\hskip-3pt\int_0^p\nh Y(t)\,dt=\int_0^p\nh f(t)\,dt+pE$ due to
periodicity of $\,Y\nh$ and constancy of both $\,C\,$ and $\,E$, so that the
second arrow in (\ref{cmp}) takes $\,Y\nh$ to
$\,Z=-(2C)^{-\nnh1}\nnh[\int_0^p\nh f(t)\,dt+pE]$. 
Applying to this $\,Z\,$ the last arrow of (\ref{cmp}), we see that 
$\,d\hh\varPhi\nnh_{\hs h+A}\w(\fh+E)
=-e^{-pC}\nh(2C)^{-\nnh1}\nnh[\int_0^p\nh f(t)\,dt+pE]$, and so
$\,d\hh\varPhi\nnh_{\hs h+A}\w$ is manifestly 
surjective onto $\,\dg\nh_m\w$. The pre\-image $\,\varPhi^{-\nnh1}\nh(e^{-pC})\,$
is thus an in\-fi\-nite-di\-men\-sion\-al sub\-man\-i\-fold of the manifold
formed by our $\,\fh+E$, with the tangent space at $\,h+A\,$ equal to
$\,\mathrm{Ker}\,d\hh\varPhi\nnh_{\hs h+A}\w$, and hence consisting of all
$\,\fh+E\,$ with $\,E=0\,$ and $\,\int_0^p\nh f(t)\,dt=0$. See
\cite[p.\ 30]{lang}.

The hypotheses of Lemma~\ref{smthg} are now satisfied by the circle
$\,Q=S^1\nh$, the space $\,\mathcal{Z}\,$ of all diagonal $\,m\times m\,$
matrices, its sub\-space $\,\mathcal{X}\,$ consisting of trace\-less ones,
the points $\,y=h+A\,$ and $\,z=e^{-pC}\nh$, and our $\,\varPhi\,$ along with
its domain $\,\mathcal{U}\nh'\nh$, treated as a subset of
$\,C^k\nh(S^1\nh,\bbR)\times\mathcal{X}\,$ via the identification of each
$\,\fh+E\,$ with the pair $\,(\fh\nh,E)$. For the smooth retraction 
$\,\ret\,$ arising from Lemma~\ref{smthg}, $\,\ve\,$ near $\,0\,$
in $\,\bbR$, and any given $\,\fh\in C^\infty\nh(S^1\nh,\bbR)\,$ with 
$\,\int_0^p\nh f(t)\,dt=0$, the curve 
$\,\ve\mapsto\ret(h+\ve\fh,A)\,$ lies in the pre\-image
$\,\varPhi^{-\nnh1}\nh(e^{-pC})$, consists of func\-tion-ma\-trix pairs having
a smooth function component, and its velocity vector at $\,\ve=0\,$ is
$\,\fh\,$ (as one sees applying $\,d/d\ve\,$ and noting that the differential
of $\,\ret\,$ at $\,y=h+A\,$ equals the identity when restricted to the
tangent space of $\,\varPhi^{-\nnh1}\nh(e^{-pC})$). Since $\,\fh\,$ was just
any smooth function $\,S^1\to\bbR$ with $\,\int_0^p\nh f(t)\,dt=0$, such
curves realize the in\-fi\-nite-di\-men\-sion\-al manifold of $\,C^\infty$
functions named in our assertion. In addition, the curve
$\,\ve\mapsto\op^{-\nnh1}\nh(\ret(h+\ve\fh,A))$ consists of smooth 
ma\-trix-val\-ued functions $\,B\,$ due to regularity of solutions for the
differential equation $\,\dot B+B^2\nh=\fh+E$. 
Since $\,(\gy_1\w,\dots,\gy_m\w)=e^{-pC}$ was an arbitrary $\,m$-tuple
with (\ref{pol}), this completes the proof.
\end{proof}
While, as we just showed, condition (\ref{pol}) is sufficient for 
$\,(\gy_1\w,\dots,\gy_m\w)\,$ to lie in the image of $\,\mathcal{X}_m\w$ under
the mapping $\,(\fh\nh,\gy_1\w,\dots,\gy_m\w)\mapsto(\gy_1\w,\dots,\gy_m\w)$,
a weaker version of (\ref{pol}) is also {\it necessary\/} for it: this
version, allowing some $\,\gy_i\w$ to equal $\,1$, states that
$\,\{\gy_1\w,\dots,\gy_m\w\}\subseteq(0,\infty)\,$ does not have the form
$\,\{\gy\}\,$ or $\,\{\gy,\gy\nh^{-\nnh1}\nh\}\,$ with any $\,\gy>0$. To see its necessity,
write $\,B=(b_1\w,\dots,b_m\w)\,$ and $\,A=(a_1\w,\dots,a_m\w)$, so that the
first equality in (\ref{prs}) amounts to $\,\dot b_i\w\nh+b_i^2=\fh+a_i\w$ for
$\,i=1,\dots,m$. Next,
\begin{equation}\label{aij}
\mathrm{if\ }\,\gy_i\w\mathrm{\ equals\ }\,\gy_j\w\mathrm{\ or\
}\,\gy_j^{-\nnh1}\mathrm{\ for\ some\ distinct\ }\,i,j\mathrm{,\ then\
}\,a_i\w=a_j\w\hh.
\end{equation}
Namely, were this not the case, so
that $\,a_i\w\ne a\hn_j\w$, while the integrals from $\,0\,$ to $\,p$ of
$\,b_i\w$ and $\,b\hn_j\w$ are equal (or, opposite), cf.\ (\ref{prs}), the
difference $\,b_i\w\nh-b\hn_j\w$ (or, the sum $\,b_i\w\nh+b\hn_j\w$) would be the
derivative $\,\dot\chi\,$ of some periodic function $\,\chi\,$ and the
equality $\,(b_i\w\nh-b\hn_j\w)\hskip-2.2pt\dot{\phantom o}\nh+
(b_i\w\nh+b\hn_j\w)(b_i\w\nh-b\hn_j\w)=a_i\w-a\hn_j\w$ would give
$\,\ddot\chi=a_i\w-a\hn_j\w\,$ wherever $\,\dot\chi=0$ or, respectively, 
$\,[(b_i\w\nh-b\hn_j\w)e^\chi]\hskip-1.5pt\dot{\phantom o}\nh
=(a_i\w-a\hn_j\w)e^\chi$. As $\,a_i\w-a\hn_j\w$ is now a nonzero constant, in the
former case the critical points of $\,\chi\,$ would all be strict local 
maxima, or strict local minima, and in the latter
$\,(b_i\w\nh-b\hn_j\w)e^\chi\,$ would be strictly monotone, both of which
contradict periodicity, thus proving (\ref{aij}). 
Combining (\ref{aij}) with second equality in (\ref{prs}), we now see that
$\,|\nh\log\gy_1\w\hn|,\dots,|\nh\log\gy_m\w\hn|\,$ cannot be all equal,
as that would give $\,a_1\w=\ldots=a_m\w$, whereas $\,A\,$ in (\ref{prs}) is
nonzero and trace\-less.

\section{Proof of Theorem~\ref{maith}: existence}\label{po} 
\setcounter{equation}{0}
The argument presented in this section proves a special case of an assertion
established in \cite[Section 9]{derdzinski-roter-10}. For the reader's
convenience we chose  to proceed as below, rather than cite 
\cite{derdzinski-roter-10}, since this simplifies the exposition.

Existence in Theorem~\ref{maith} will follow once we show that, for suitable
$\,\fh,p,n,$ $\mv\nh,\lr,A\,$ with the properties listed at the beginning of
Section~\ref{mm}, where $\,n\ge5$ and the
metric signature of $\,\lr\,$ are {\it arbitrary}, and for
$\,\gp,\hm\,$ appearing in (\ref{ism}),
\begin{equation}\label{sgp}
\begin{array}{l}
\mathrm{some\ sub\-group\ }\,\Gm\subseteq\gp\,\mathrm{\ acts\ on\ 
}\,\hm\,\mathrm{\ freely\ and\ properly\ dis}\hyp\\
\mathrm{con\-tin\-u\-ous\-ly\ with\ a\ compact\ quotient\ manifold\
}\,M=\hm\nnh/\hh\Gm\nh.
\end{array}
\end{equation}
To choose $\,\fh,p,n,\mv\nh,\lr,A\,$ satisfying the conditions named in
Section~\ref{mm}, along with additional objects $\,\rc,B,\lz,\Lm,\gkp,\Gm\,$ 
needed for our argument, we let $\,n\ge5\,$ and 
$\,p,\rc\in(0,\infty)\,$ be completely arbitrary, and denote by $\,\lr\,$ a
pseu\-\hbox{do\hs-}Euclid\-e\-an inner product of any signature in
$\,\mv\nh=\bbR^{\hskip-.6ptn-2}\nnh$, making the standard basis orthonormal.
Lemma~\ref{polyn} allows us to fix a $\,\mathrm{GL}\hh(m,\bbZ)\,$ 
polynomial $\,P\nh$, where $\,m=n-2$, the complex roots
$\,\gy_1\w,\dots,\gy_m\w$ of which are all real, distinct, and satisfy
(\ref{pol}). Theorem~\ref{pospe}, for these $\,\gy_1\w,\dots,\gy_m\w$, yields
$\,\fh\nh,A\,$ and a curve
$\,t\mapsto B(t)\in\dg\nh_m\w\subseteq\text{\rm End}\hs(\mv)$, with
\begin{equation}\label{inf}
\mathrm{in\-fi\-nite}\hyp\mathrm{di\-men\-sion\-al\ freedom\ of\
choosing\ }\,\fh.
\end{equation}
Next, we let $\,\lz\,$ be the $\,(n-2)$-di\-men\-sion\-al vector space of
all solutions $\,u:\bbR\to\mv\hs$ to the differential equation 
$\,\dot u(t)=B(t)\hh u(t)$, with the {\it translation operator\/} 
$\,T:\lz\to\lz$ given by $\,(Tu)(t)=u(t-p)$. Note that
$\,\lz\subseteq\xe\hs$ for $\,\xe\hs$ which was defined in Section~\ref{mm}
along with a linear isomorphism $\,T:\xe\to\xe\nh$, and our $\,T\,$ is the restriction of that one to $\,\lz$. According to
\cite[Remark 4.2]{derdzinski-roter-10} and 
(\ref{prs}), our $\,T:\lz\to\lz\,$ has the spectrum 
$\,\gy_1\w,\dots,\gy_m\w$, so that $\,P\,$ is its characteristic polynomial 
and, by (\ref{glz}), $\,T(\Lm)=\Lm\,$ for some lattice $\,\Lm\,$ in $\,\lz$.
(As $\,\gy_1\w,\dots,\gy_m\w$ are all distinct, they uniquely determine the 
algebraic equivalence type of $\,T\nh$.) For later reference, let us also note
that
\begin{equation}\label{tko}
T^k\nh\ne\mathrm{Id}\hskip9pt\mathrm{for\ all\
}\,\,k\in\bbZ\smallsetminus\{0\}\hh,
\end{equation}
for $\,T:\lz\to\lz$, as its spectrum $\,\{\gy_1\w,\dots,\gy_m\w\}\,$ is
contained in $\,(0,\infty)\smallsetminus\{1\}$. Now
$\,\gkp=\{0\}\times\bbZ\rc\times\Lm\,$ is both a subset of
$\,\gp=\bbZ\times\bbR\times\xe\hs$ and a 
lattice in the vector space $\,\{0\}\times\bbR\times\lz\,$ while,
due to self-ad\-joint\-ness of each $\,B(t)$,
\begin{equation}\label{ouw}
\varOmega(u,w)\,=\,0\hskip7pt\mathrm{whenever\ }\,\,u,w\in\lz\hh.
\end{equation}
Finally, we denote by 
$\,\Gm$ the sub\-group of $\,\gp\,$ generated by $\,\gkp\,$ and the element 
$\,(1,0,0)$. Then $\,\Gm\nh$, as a subset of
$\,\gp=\bbZ\times\bbR\times\xe\nh$, equals $\,\bbZ\times\bbZ\rc\times\Lm$, and
\begin{equation}\label{acg}
\begin{array}{l}
\mathrm{each\ }\,\,(k,\ell\rc,u)\,\in\,\Gm\,
=\,\,\bbZ\times\bbZ\rc\times\Lm\,\,\mathrm{\ acts\,\ on\ 
}\hm\,=\,\rto\nh\times\mv\,\mathrm{\ by}\\
(k,\ell\rc,u)\cdot(t,s,v)
=(t+kp,\hs s+\ell\rc-\lg\dot u(t),2v+u(t)\rg,\hs v+u(t))\hs.
\end{array}
\end{equation}
In fact, due to (\ref{cnj}) and $\,T$-in\-var\-i\-ance of $\,\Lm$, the 
conjugation by $\,(1,0,0)\,$ maps $\,\gkp\,$ onto itself, and any element of 
$\,\Gm\nh$, being a finite product of factors from the set 
$\,\gkp\cup\{(1,0,0),(1,0,0)^{-1}\}$, {\it equals a power of\/ $\,(1,0,0)\,$
times an element of\/} $\,\gkp$. However, by (\ref{ope}), 
$\,(k,0,0)\cdot(\ell,0,0)=(k+\ell,0,0)$, and so $\,(1,0,0)^k\nh=(k,0,0)\,$ 
if $\,k\in\bbZ$. The last italicized phrase, combined with (\ref{ope}) 
and (\ref{act}), now yields (\ref{acg}).

The action (\ref{acg}) is free: if $\,(k,\ell\rc,u)\cdot(t,s,v)=(t,s,v)$, the
resulting equalities $\,kp=\ell\rc-\lg\dot u(t),2v+u(t)\rg=0\,$ and
$\,u(t)=0\,$ give $\,k=0$, while the first-order linear differential equation 
$\,\dot u=Bu\,$ implies that $\,u=0$, and so $\,\ell=0\,$ as well. 

In view of the reg\-u\-lar-de\-pend\-ence theorem for ordinary differential
equationsrr,
\begin{equation}\label{dif}
\bbR\times\lz\ni(t,w)\mapsto(t,w(t))\in\bbR\times\mv\,\mathrm{\ is\ a\
dif\-feo\-mor\-phism,}
\end{equation}
since that theorem guarantees smoothness of the inverse of (\ref{dif}). We 
now use Remark~\ref{stbil} to conclude that (\ref{acg}) is properly
dis\-con\-tin\-u\-ous: if the sequences
$\,(t\nh_j\w,s\nh_j\w,v\nh_j\w)\,$ and
$\,(k\nh_j\w,\ell\nnh_j\w\rc,u\nh_j\w)\cdot(t\nh_j\w,s\nh_j\w,v\nh_j\w)\,$
both converge, (\ref{acg}) gives convergence of $\,k\nh_j\w$ and
$\,u\nh_j\w(t\nh_j)\w$. The former makes $\,k\nh_j\w$ eventually constant, the
latter leads, by (\ref{dif}), to convergence of $\,u\nh_j\w$, and hence its
ultimate constancy (implying via (\ref{acg}) the same for $\,\ell\nnh_j\w$),
as $\,u\nh_j\w\in\Lm\,$ and the lattice $\,\Lm\subseteq\lz\,$ is discrete,
cf.\ (\ref{dsc}).

Finally, compactness of the quotient manifold $\,M=\hm\nnh/\hh\Gm\,$ in
(\ref{sgp}) follows since $\,\hm\,$ has a compact subset $\,K\,$ 
intersecting every orbit of $\,\Gm\nh$. Namely, we may set 
$\,K=\{(t,s,v):s\in[\hs0,\rc\hs]\,\mathrm{\ and\ }\,(t,v)\in K'\}$, where
$\,K'$ is the image under (\ref{dif}) of $\,[\hs0,p\hs]\times\hk$, with a
compact set $\,\hk\subseteq\lz\,$ chosen so as to intersect every orbit of the
lattice $\,\Lm\,$ acting on $\,\lz\,$ by vec\-tor-space translations. Any 
$\,(t,s,v)\in\hm\,$ can be successively modified by elements of $\,\Gm\hs$
acting on it, so as to eventually end up in $\,K\nh$. First, 
$\,(k,0,0)\in\Gm\hs$ with $\,kp\in[-t,p-t\hs]$, applied to $\,(t,s,v)$, 
allows us to assume that $\,t\in[\hs0,p\hs]$. For the pair $\,(t,w)\,$
arising as the pre\-im\-age under (\ref{dif}) of the $\,(t,v)$ component of
this new $\,(t,s,v)$, and suitably selected $\,u\in\Lm$, one has
$\,w+u\in\hk$, due to our choice of $\,\hk$. Now, by (\ref{acg}) with
$\,v=w(t)$,
\[
(0,0,u)\cdot(t,s,v)=(t,s-\lg\dot u(t),2v+u(t)\rg,\hs w(t)+u(t))\hh,
\]
that is, $\,(0,0,u)\cdot(t,s,v)=(t'\nh,s'\nh,v')\,$ for some
$\,(t'\nh,v')\in K'$ and $\,s'\nh\in\bbR$. Choosing $\,\ell\in\bbZ\,$ such
that $\,s'\nh+\ell\rc\in[\hs0,\rc\hs]$, we obtain
$\,(0,\ell\rc,0)\cdot(t'\nh,s'\nh,v')\in K$.

\section{Proof of Theorem~\ref{maith}: further conclusions}\label{pc} 
\setcounter{equation}{0}
For $\,\hm\nh,p\,$ and $\,\Gm=\hs\bbZ\times\bbZ\rc\times\Lm\,$ as in the last
section, the surjective sub\-mer\-sion $\,\hm=\rto\nh\times\mv\to\bbR\,$ 
sending $\,(t,s,v)\,$ to $\,t/p\,$ is, by (\ref{acg}), equi\-var\-i\-ant 
relative to the homo\-mor\-phism
$\,\Gm=\hs\bbZ\times\bbZ\rc\times\Lm\ni(k,\ell\rc,u)\to k\in\bbZ\,$ along with 
the actions of $\,\Gm\,$ on $\,\hm\,$ and $\,\bbZ\,$ on $\,\bbR$, so that it
descends to a surjective sub\-mer\-sion 
$\,M=\hm\nnh/\hh\Gm\to\bbR/\bbZ=S^1$ which, according to Remark~\ref{bndpr},
is a bundle projection. Since the homo\-mor\-phism $\,\Gm\to\bbZ\,$ has the
kernel $\,\gkp=\{0\}\times\bbZ\rc\times\Lm$, the fibre of this projection
$\,M\to S^1$ over the $\,\bbZ$-co\-set of $\,t/p\,$ may be identified
with the quotient $\,\hm\nnh_t\w/\gkp$, where 
$\,\hm\nnh_t\w=\{t\}\times\bbR\times\mv\nh$. Restricted to $\,\gkp\,$ and
$\,\hm\nnh_t\w$, (\ref{acg}) is given by
\begin{equation}\label{rst}
(0,\ell\rc,u)\cdot(t,s,v)
=(t,\hs s+\ell\rc-\lg\dot u(t),2v+u(t)\rg,\hs v+u(t))\hs,
\end{equation}
with fixed $\,t\in\bbR$. By (\ref{dif}), the restriction of (\ref{mpg}) to
$\,\{t\}\times\bbR\times\lz\,$ is a dif\-feo\-mor\-phism
onto $\,\hm\nnh_t\w$, and so its
$\,\gkp\nh$-equi\-var\-i\-ance, immediate from $\,\gp\nh$-equi\-var\-i\-ance,
means that, when we use it to identify $\,\hm\nnh_t\w$ with
$\,\{t\}\times\bbR\times\lz$, and hence also with
$\,\hp'\nh=\{0\}\times\bbR\times\lz\,$ (a sub\-group of $\,\gp\,$ containing
$\,\Sigma$), the restriction of (\ref{act}) to
$\,\hp'\nh\times\hm\nnh_t\w$ becomes the action of $\,\hp'$ on
itself via left translations. By (\ref{ope}) with $\,k=\ell=0$ and
(\ref{ouw}), $\,\hp'$ is an Abel\-i\-an sub\-group of $\,\gp$, and the
resulting group operation in $\,\hp'$ coincides with addition in the vector
space $\,\{0\}\times\bbR\times\lz\nh$. Since 
$\,\gkp=\{0\}\times\bbZ\rc\times\Lm$ is a lattice in
$\,\{0\}\times\bbR\times\lz$, cf.\ the lines preceding (\ref{ouw}), this
shows that the fibre $\,\hm\nnh_t\w/\gkp\,$ is a torus, which
makes $\,M\nh$, with the projection $\,M\to S^1$ described above, 
{\it a torus bundle over the circle}.

The torus bundle $\,M\to S^1$ is nontrivial: (\ref{tko}) combined with
\cite[Theorem 5.1(f)]{derdzinski-roter-10}, implies that $\,\Gm\,$ has no
Abel\-i\-an subgroup of finite index, so that $\,M\,$ cannot be
dif\-feo\-mor\-phic to a torus, or even covered  by a torus.

Geodesic completeness of $\,M\nh$, and its lack of local homogeneity, are
immediate from (\ref{gco}), while the claim about an
in\-fi\-nite-di\-men\-sion\-al moduli space of the lo\-cal-isom\-e\-try types
is an obvious consequence of Remark~\ref{infdm} and (\ref{inf}).


\end{document}